\documentclass[11pt,reqno]{amsart}
\usepackage{amsfonts,amsmath,amssymb}
\newtheorem{Theorem}{Theorem}
\newtheorem{prop}{Proposition}

\def\beq#1#2\eeq{%
        \begin{equation}%
        \label{#1}%
            #2%
        \end{equation}%
    }

\usepackage{color}
\usepackage{graphicx}
\usepackage{epstopdf}

\title[Lyapunov and Conway]{Growth of values of binary quadratic forms and Conway rivers}

\author{ K. Spalding}\address{Department of Mathematical Sciences,
Loughborough University, Loughborough LE11 3TU, UK}
\email{K.Spalding@lboro.ac.uk}

\author{A.P. Veselov}
\address{Department of Mathematical Sciences,
Loughborough University, Loughborough LE11 3TU, UK  and Moscow State University, Moscow 119899, Russia}
\email{A.P.Veselov@lboro.ac.uk}

\begin{document}

\maketitle

\begin{abstract}
We study the growth of the values of integer binary quadratic forms $Q$ on a binary planar tree as it was described by Conway. We show that the corresponding Lyapunov exponents $\Lambda_Q(x)$ as a function of the path determined by $x\in \mathbb RP^1$ are twice the values of the corresponding exponents for the growth of Markov numbers \cite{SV}, except for the paths corresponding to the Conway river, when $\Lambda_Q(x)=0.$ The relation with the Galois result about pure periodic continued fractions is explained and interpreted geometrically.
\end{abstract}


\section{Introduction}

In his book ``The Sensual (Quadratic) Form" Conway \cite{Conway} described the following ``topographic" way to visualise the values of a binary quadratic form
\begin{equation}
\label{Q}
Q(x, y) = ax^2 + hxy + by^2
\end{equation}
with integer coefficients $a,b,c$ on the integer lattice $(x,y) \in \mathbb{Z}^2.$

Following Conway we introduce the notions of the {\it lax} vector as a pair $(\pm v), v \in \mathbb Z^2$, and of the \emph{superbase} of the integer lattice $\mathbb Z^2$ as a triple of lax vectors $(\pm e_1, \pm e_2, \pm e_3)$ such that $(e_1, e_2)$ is a basis of the lattice and
		\begin{equation*}
		e_1 + e_2 + e_3 =  0.
		\end{equation*}
It is easy to see that every basis can be included in exactly two superbases, which we can represent using the binary tree embedded in the plane (see Figure 1).
The lax vectors live in the complement to the tree (we show only one representative of them), while the superbases correspond to the vertices.
	
%
		
\begin{figure}[h]
\begin{center}
 \includegraphics[height=48mm]{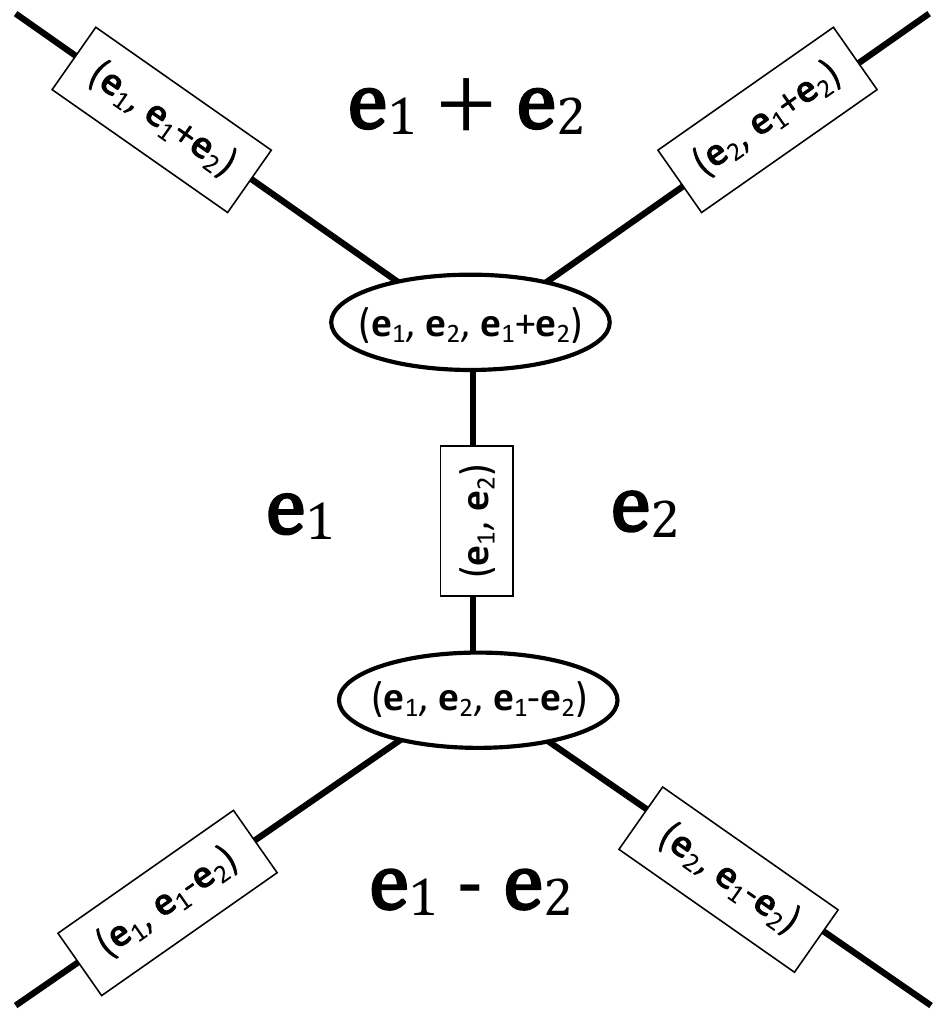}  \hspace{8pt}  \includegraphics[height=48mm]{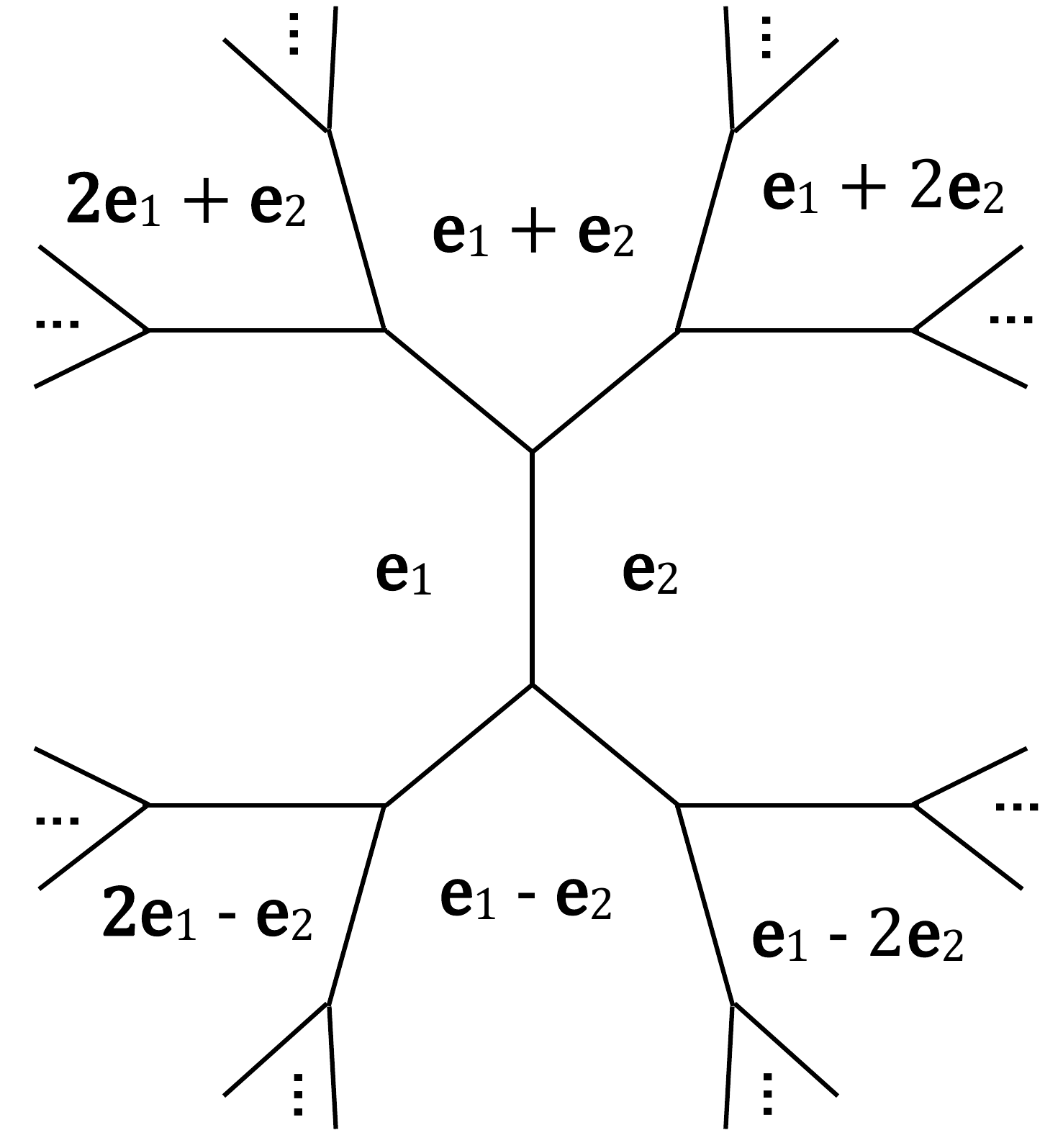}
\caption{\small The superbase tree and its version with only faces labelled.}
\end{center}
\end{figure}

%
%
	
By taking values of the form $Q$ on the vectors of the superbase, we get what Conway called the {\it topograph} of $Q$ containing the values of $Q$ on all primitive lattice vectors.
In particular, if $\mathbf{e}_1=(1,0), \mathbf{e}_2=(0,1), \mathbf{e}_3=-(1,1)$ we have the values 
$$
Q(\mathbf{e}_1)=a, \,\, Q(\mathbf{e}_2)=b, \,\, Q(\mathbf{e}_3)=c:=a+b+h.
$$
One can construct the topograph of $Q$ starting from this triple using the following property of any quadratic form, which Conway called the {\it arithmetic progression rule} and is known also in geometry as the {\it parallelogram rule}:
\begin{equation}
\label{apr}
Q(\mathbf{u}+\mathbf{v})+Q(\mathbf{u}-\mathbf{v})=2(Q(\mathbf{u})+Q(\mathbf{v})), \quad \mathbf{u},\mathbf{v} \in \mathbb R^2.
\end{equation}
Note that quadratic forms (in any dimension) can be characterised as the degree 2 homogeneous functions satisfying relation (\ref{apr}).
This leads to the following relation on the topograph of $Q$ (see Fig. 2), which can be used to compute the primitive values of $Q$ recursively.
\begin{figure}[h]
	\centering
	\includegraphics[scale=0.25]{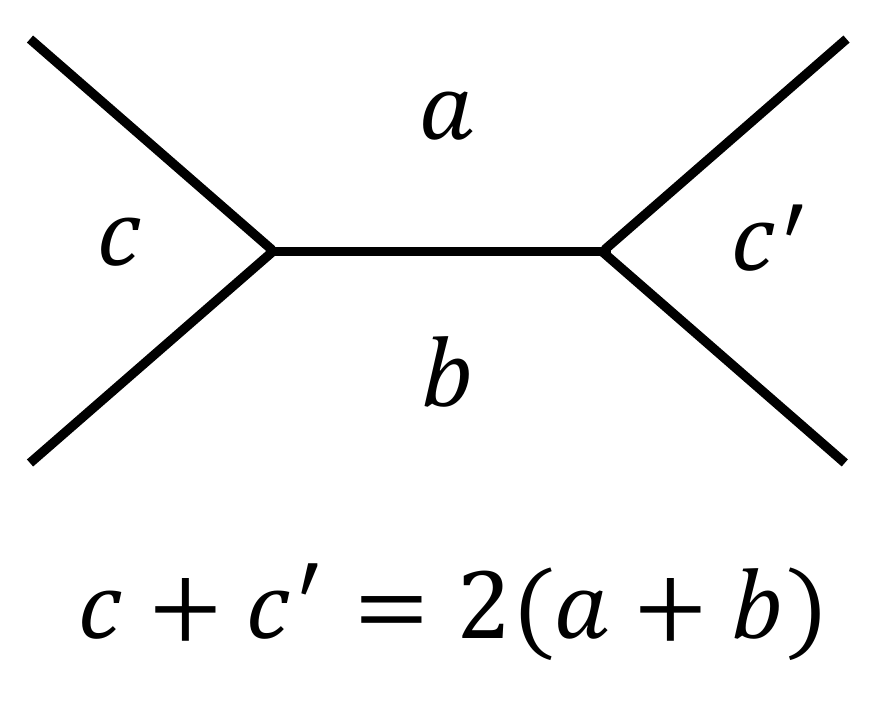}
	\caption{Arithmetic progression rule for values of quadratic forms.}
	\end{figure}

In particular, for the standard quadratic form $Q(x,y)=x^2+y^2$ we have the tree of which half is shown on the left of Fig. 3.
On the right side of Fig. 3 we show the positive part of the corresponding {\it Farey tree}, where at each vertex we have the fractions $\frac{p}{r}, \, \frac{q}{s}$ and its {\it Farey mediant}
$\frac{p+q}{r+s}$ (see e.g. \cite{Hatcher,SV}), starting with $\frac{1}{0}, \, \frac{0}{1}.$

Using the Farey tree we can identify the infinite paths $\gamma$ on a binary tree with real numbers $\xi \in [0,\infty]$ (more details in the next section).
For example, for the golden ratio $\xi=\varphi:=\frac{\sqrt{5}+1}{2}$ we have the Fibonacci path shown in bold on both trees of Fig. 3.

We would like to study the growth of the values of $Q$ along the path $\gamma_\xi$. More precisely, define
\beq{defL}
\Lambda_Q(\xi)=\limsup_{n\to\infty}\frac{\ln |Q_n(\xi)|}{n},
\eeq
where $$|Q_n(\xi)| = \max(|a_n(\xi)|, |b_n(\xi)|, |c_n(\xi)|)$$ with $(a_n(\xi), b_n(\xi), c_n(\xi))$ being the $n$-th topographic triple on the path $\gamma_\xi.$

For example, for the Fibonacci path with $\xi=\varphi$ we have $q_n=F_{2n}$ being every second Fibonacci number with the growth
$$
\limsup_{n\to\infty}\frac{\ln |Q_n(\xi)|}{n}=\lim_{n\to\infty}\frac{\ln F_{2n}}{n}=\ln \varphi^2=2 \ln \varphi.
$$
We will show that a similar result is true for any positive binary quadratic form $Q$, namely that
\beq{LQ}
\Lambda_Q(\xi)=2 \Lambda(\xi),
\eeq
where $\Lambda(\xi)$ is the function introduced in \cite{SV} describing the growth of the Markov numbers, or, equivalently, the growth of the monoid $SL_2(\mathbb N)$. 

The monoid $SL_2(\mathbb N)$ consists of matrices from $SL_2(\mathbb Z)$ with non-negative entries. These can be seen on the edges of the Farey tree if we combine two neighbouring fractions $\frac{p}{r}, \, \frac{q}{s}$ into the matrix
\beq{A}
A= \begin{pmatrix}
  p & q \\
  r & s \\
\end{pmatrix} \in SL_2(\mathbb N).
\eeq
Let $\lambda_1, \lambda_2$ be the eigenvalues of $A$, which are real numbers with $\lambda_1\lambda_2=1.$

The function $\Lambda(\xi)$ can be defined as 
\beq{defeucl}
\Lambda(\xi)=\limsup_{n\to\infty}\frac{\ln \rho(A_n(\xi))}{n},
\eeq
where $A_n(\xi) \in SL_2(\mathbb N)$ is attached to the $n$-th edge along path $\gamma_\xi$
and $$\rho(A)=max(|\lambda_1|,|\lambda_2|)$$ is the {\it spectral radius} of the matrix $A$ (see \cite{SV}).

The function $\Lambda(\xi)$ can be extended to $\xi \in {\mathbb R}P^1$ and has very peculiar properties: it is discontinuous everywhere, takes all real values from $[0, \ln \varphi]$ and is invariant under the projective action of $GL_2(\mathbb Z)$ on ${\mathbb R}P^1$
(see \cite{SV}).

\begin{figure}[h]
\begin{center}
 \includegraphics[height=48mm]{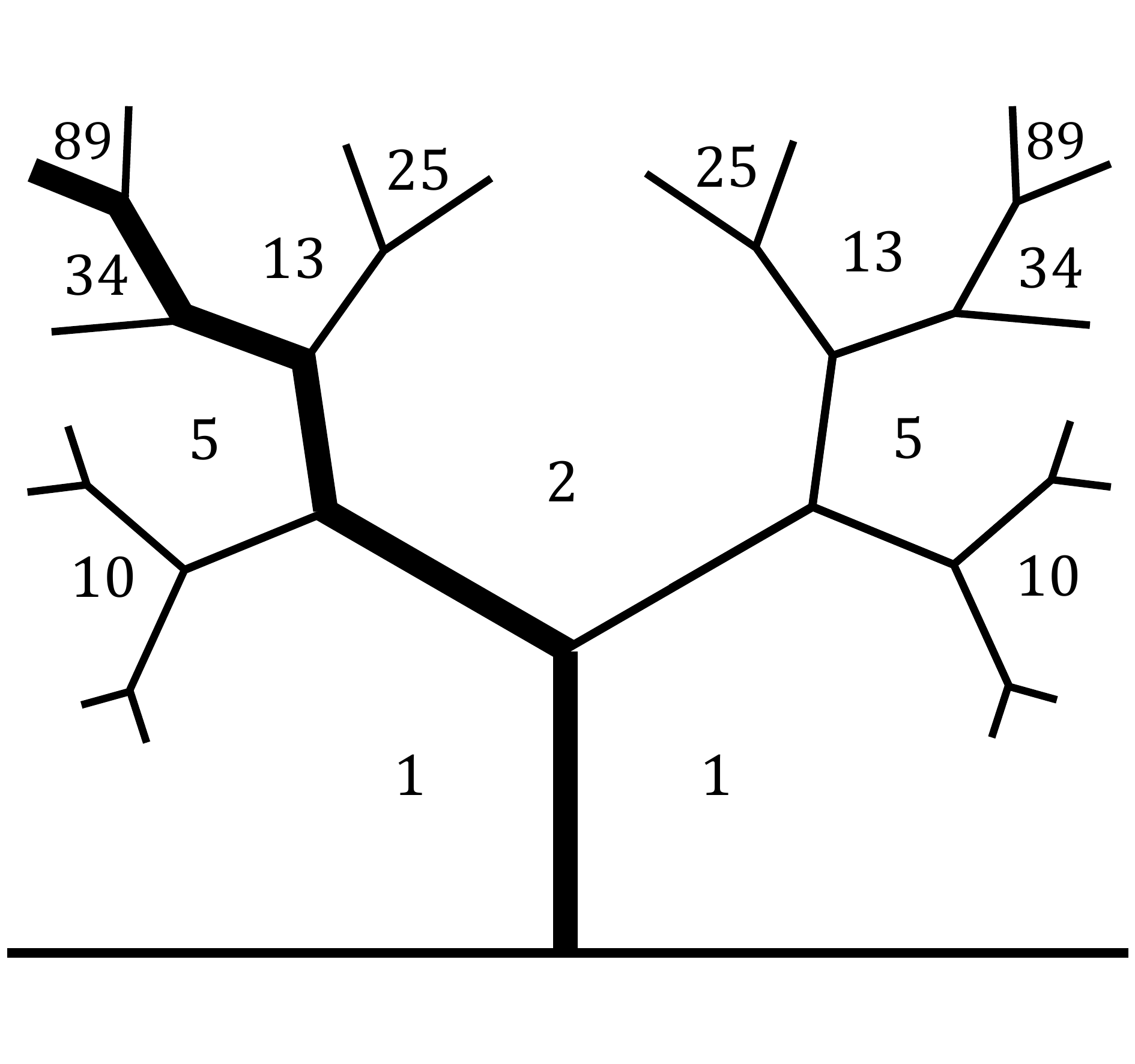}  \hspace{8pt}  \includegraphics[height=48mm]{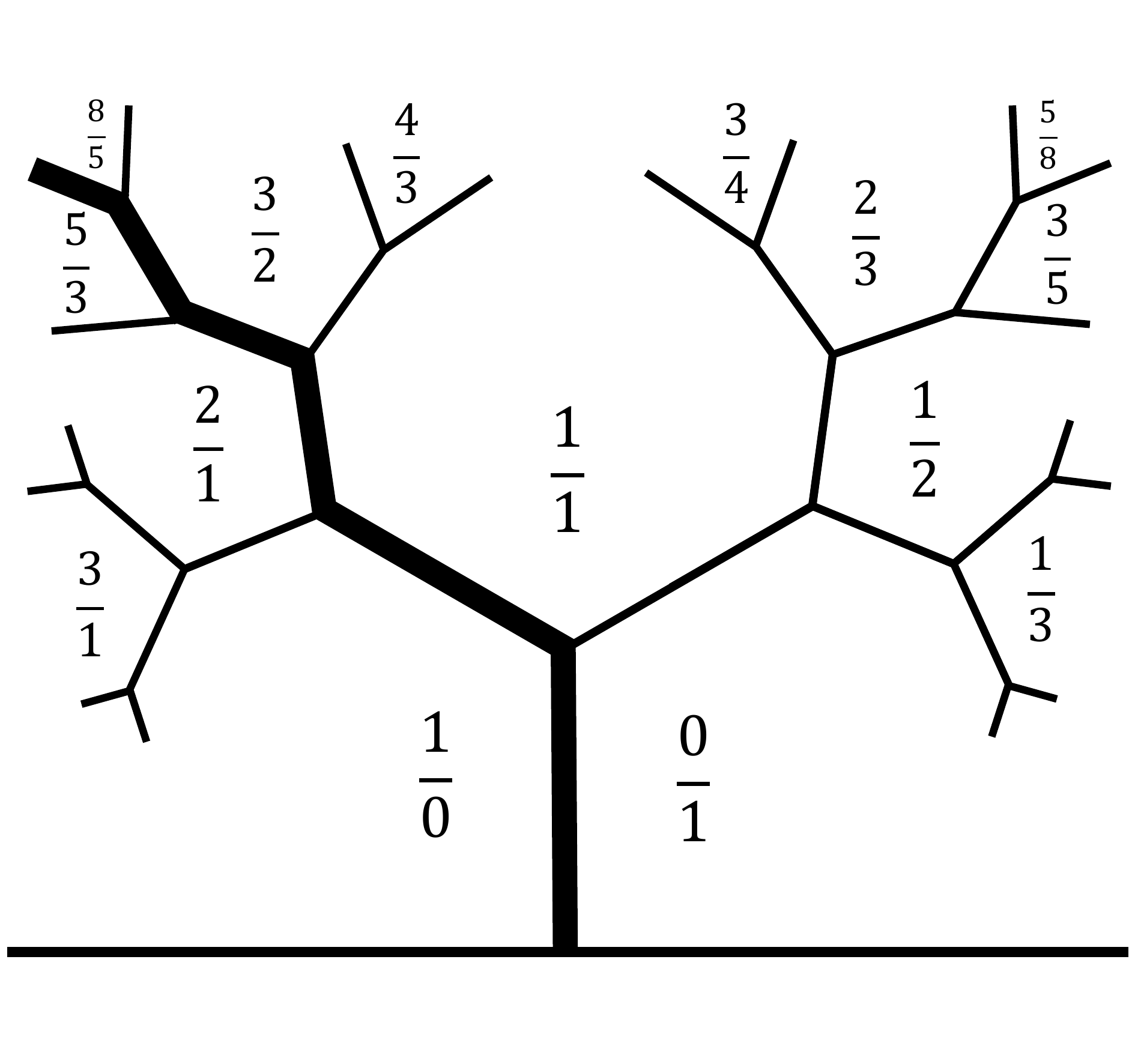}
\caption{\small Topograph of $Q=x^2+y^2$ and the corresponding positive part of the Farey tree with marked Fibonacci path.}
\end{center}
\end{figure}

The situation is different for indefinite binary quadratic forms. The reason is the existence of what Conway \cite{Conway} called the {\it river}, which, in the case of the forms not representing zero, is an infinite path on the binary tree separating positive and negative values of the form $Q$ (see Fig. 4).

\begin{figure}[h]
\begin{center}
 \includegraphics[height=40mm]{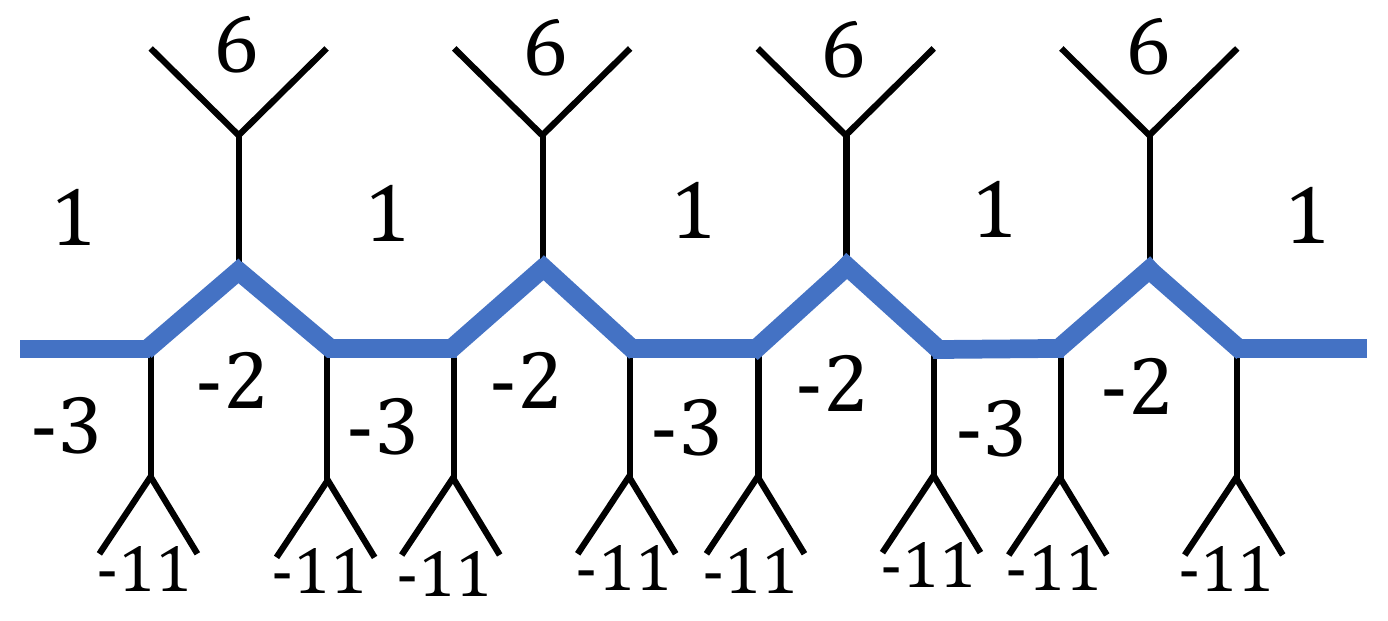}  
\caption{\small Conway river for the indefinite binary quadratic form  $Q=x^2-2xy-2y^2.$}
\end{center}
\end{figure}

Let $\alpha_{\pm}$ be the two real roots of the corresponding quadratic equation
$$
Q(\alpha,1)=0.
$$
Our main result says that for an indefinite form $Q$ not representing zero
\beq{LQ2}
\Lambda_Q(\xi)=2 \Lambda(\xi), \quad \xi \neq \alpha_{\pm}
\eeq
with $\Lambda_Q(\alpha_{\pm})=0 \neq 2 \Lambda(\alpha_{\pm}).$ 
For example, for $Q=x^2-2xy-2y^2$ we have $\alpha_{\pm}=1\pm \sqrt{3}$ with $$\Lambda(1+\sqrt{3})=\frac{1}{3} \ln (2+\sqrt{3}), \,\, \Lambda_Q(1+\sqrt{3})=0.$$
We will show that the two exceptional paths with zero growth are exactly those leading to the two ends of the Conway river (see Fig. 6 below).

We also discuss in more detail the geometry of the corresponding exceptional paths in relation to the continued fraction expansions of the quadratic irrationals $\alpha_{\pm}$. The Galois result about pure periodic continued fractions \cite{Galois} naturally appears in this way.

In the case when indefinite form $Q$ does represent zero (which means that its discriminant is total square) the roots $\alpha_{\pm}$ are rational, so $\Lambda(\alpha_{\pm})=0$ and $\Lambda_Q(\xi)=2\Lambda(\xi)$ for all $\xi$ in this case. The same is true for the semidefinite forms with discriminant zero.

\section{Paths in a binary planar tree, negative continued fractions and Galois theorem}

Let us first describe in more detail the correspondence between $\xi\in [0,\infty]$ and paths $\gamma$ in the planar binary rooted tree shown on the right of Fig. 3. Indeed, the corresponding Farey tree has a unique path $\gamma_\xi$ for any irrational positive $\xi$ such that the limit of the adjacent fractions is $\xi$ (for rational $\xi$ such a path is finite and leads to the corresponding fraction).

In terms of the corresponding continued fraction expansion
\begin{equation*}
\xi = c_0 + \cfrac{1}{c_1 
          + \cfrac{1}{c_2
          + \ddots } } :=\left[ c_0, c_1, c_2, \ldots \right], \quad c_i\in \mathbb Z_+,
\end{equation*}
the path $\gamma_\xi$ starts from the root and can be described as $c_0$ left-turns on the tree, followed by $c_1$ right-turns, followed by $c_2$ left-turns, etc.

For every oriented edge $E$ adjacent to two Farey fractions $\frac{a}{c}$, $\frac{b}{d}$ we assign the unimodular matrix
$$
A_E=\begin{pmatrix}
  a & b \\
  c & d \\
\end{pmatrix}.
$$
Then the corresponding matrix $A_n(\gamma),\, \gamma=\gamma_\xi$ is the product of the first $n$ matrices along the path $\gamma$:
\beq{mata}
A_n(\gamma)= L^{c_0}R^{c_1}L^{c_2}R^{c_3} \ldots 
\eeq
where 
$$
L=\begin{pmatrix}
  1 & 1 \\
  0 & 1 \\
\end{pmatrix}, \quad R= \begin{pmatrix}
  1 & 0 \\
  1 & 1 \\
\end{pmatrix}.
$$
The matrices $L$ and $R$ freely generate the monoid $SL_2(\mathbb N)$, which is the positive part of the group $SL_2(\mathbb Z)$, acting on the planar binary rooted tree by left and right turns respectively.

We would like to extend this correspondence to the full binary planar tree, which is known to be the dual tree for the Farey tesselation (see e.g. Hatcher \cite{Hatcher} and left side of Figure 5).

\begin{figure}[h]
\begin{center}
\includegraphics[height=60mm]{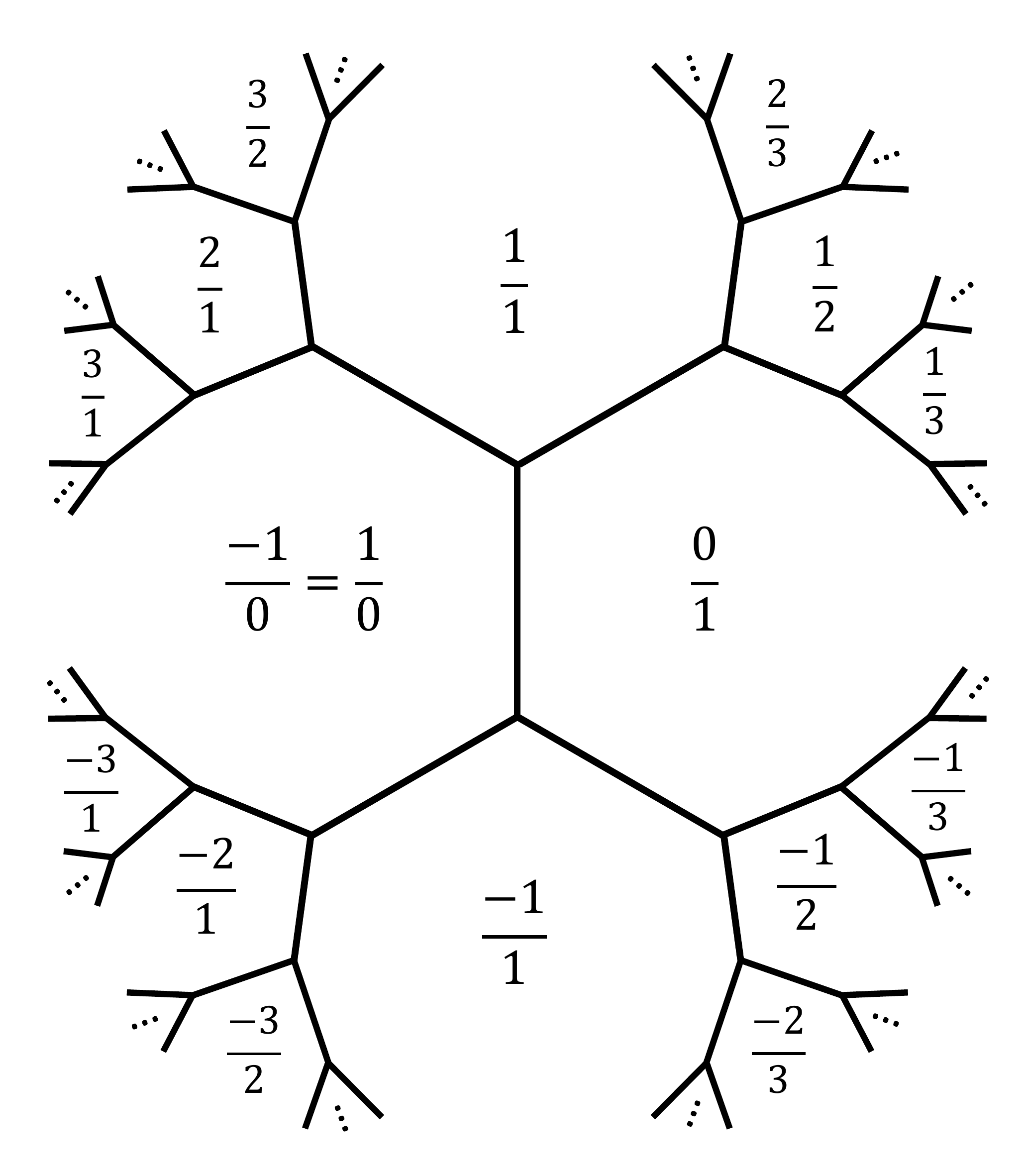}  \hspace{20pt}  \includegraphics[height=60mm]{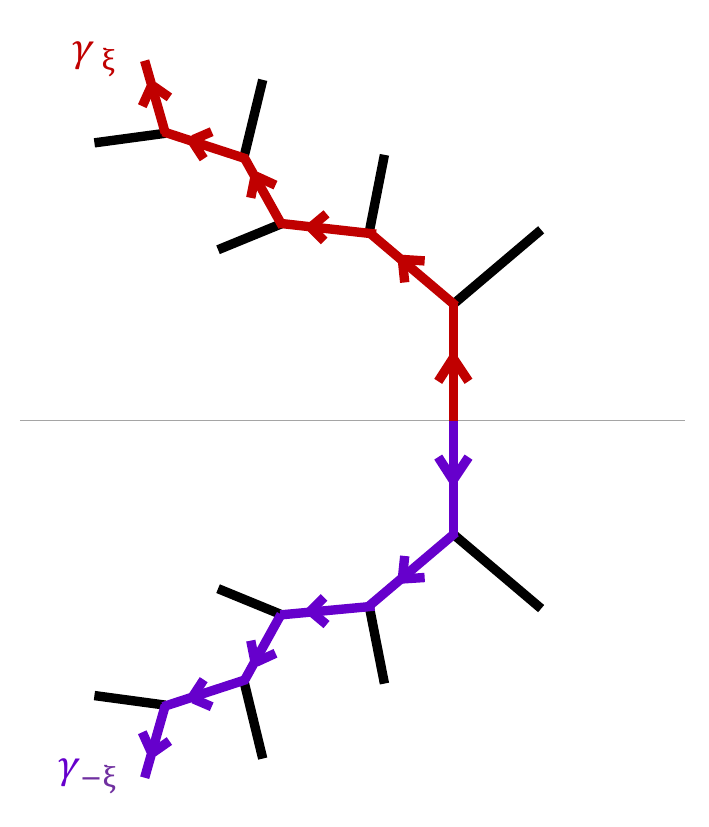}
\caption{\small Full Farey tree and reflected path $\gamma_{-\xi}.$}
\end{center}
\end{figure}

%

In particular, we see from this figure that for the reflected path $\bar\gamma=\gamma_{-\xi}$ the corresponding matrix is 
$$
A_n({\bar\gamma})=L^{-c_0}R^{-c_1}L^{-c_2}R^{-c_3}\ldots,
$$
where $A_n(\gamma)$ is given by (\ref{mata}) and for an edge $E$ in the bottom half of the Farey tree with two adjacent Farey fractions $\frac{a}{c}$, $\frac{b}{d}$ 
we assign the matrix
$$
A_E=\begin{pmatrix}
  -a & b \\
  -c & d \\
\end{pmatrix}.
$$

To be consistent with this it is natural to consider for negative real numbers $\eta=-\xi<0$ with positive $\xi=\left[ c_0, c_1, c_2, c_3 \ldots \right], \quad c_0\geq 0, c_i>0$ the {\it negative continued fraction expansions}
\beq{negcf}
\eta=-[c_0, c_1, c_2, c_3 \ldots]=[-c_0,-c_1,-c_2,-c_3 \ldots].
\eeq
Note that one can always avoid the negative $c_i$ with $i>0$ because of the identity
\beq{negid}
-[c_0, c_1, c_2, c_3, \ldots]
= [-c_0-1, 1, c_1 - 1, c_2, c_3, \ldots],
\eeq
but we are not going to do this.

This is related to the following result discovered by the 18-year-old \'Evariste Galois.
The famous Lagrange's theorem \cite{LeVeque} says that any quadratic irrational has a periodic continued fraction expansion
$
\alpha=[a_0,\dots, a_k, \overline{b_1, \ldots, b_l}],
$
and conversely, any periodic continued fraction represents a quadratic irrational.

Less known is the following addition due to Galois \cite{Galois} characterising {\it pure periodic continued fraction expansions}
$$
\alpha=[\overline{b_1, \ldots, b_l}].
$$


{\bf Theorem (Galois).} {\it A quadratic irrational $\alpha=\frac{A+\sqrt{D}}{B}$ has a pure periodic continued fraction expansion
\begin{equation*}
\alpha =[\overline{b_1, \ldots, b_l}]
\end{equation*}
if and only if its conjugate
$\bar \alpha=\frac{A-\sqrt{D}}{B}$
satisfies the inequality
$$
-1<\bar \alpha<0.
$$
Moreover, in that case}
\begin{equation*}
\bar\alpha = - [0, \overline{b_l, \ldots, b_1}].
\end{equation*}

As we will see later, geometrically the conjugate $\bar\alpha$ determines the path going backwards along the corresponding Conway river. 


\section{Continued fraction expansion of quadratic irrationals and their conjugates}

To describe our main result we need an answer to the following natural question.
Assume that we know the continued fraction expansion of a quadratic irrational
\begin{equation}
\label{pcf}
\alpha=[a_0,\dots, a_k, \overline{b_1, \ldots, b_l}].
\end{equation}
What is the continued fraction expansion of its conjugate $\bar\alpha$?

For example, we have 
\begin{equation}
\label{ex1}
\alpha=\frac{6+\sqrt{2}}{17}=[0,2,3,\overline{2}], \quad \bar\alpha=\frac{6-\sqrt{2}}{17}=[0,3,1,\overline{2}].
\end{equation}
What is the general rule here?

We could not find the answer to this question in the literature, except in the Galois case, so we present it in this section.

Note first that we can assume that in (\ref{pcf}) $a_k \neq b_l$, because otherwise $\alpha$ can be rewritten with the period $[\overline{b_l, b_1, \ldots, b_{l-1}}]$.

\begin{prop} \label{thm:conj} Let $\alpha = [a_0, a_1, \ldots, a_k, \overline{b_1, \ldots, b_l}]$ be the continued fraction expansion of a quadratic irrational with $a_k < b_l$, and $k \geq 1.$ Then the continued fraction expansion of its conjugate is
\begin{equation}
\label{conj1}
\overline{\alpha} = [a_0, \ldots, a_{k-1}-1, 1, b_l-a_k-1, \overline{b_{l-1}, b_{l-2}, \ldots, b_1, b_l}].
\end{equation}
If $a_k > b_l$, $k \geq 1$, then 
\begin{equation}
\label{conj2}
\overline{\alpha} = \left[a_0, \ldots, a_{k-1}, a_k - b_l-1, 1, b_{l-1}-1, \overline{b_{l-2}, b_{l-3}, \ldots, b_1, b_l, b_{l-1}}\right].
\end{equation}
\end{prop}

Here it will be convenient for us to allow the zeros in the continued fraction expansion, which can be always avoided using the identities
\begin{equation} \label{eq:onezero}
\left[ c_0, \ldots, c_i, 0, c_{i+1}, c_{i+2}, c_{i+3} \ldots \right] = \left[ c_0, \ldots, c_i + c_{i+1}, c_{i+2}, c_{i+3} \ldots \right],
\end{equation}
\begin{equation} \label{eq:twozero}
\left[ c_0, \ldots, c_i, 0, 0, c_{i+1}, c_{i+2}, c_{i+3} \ldots \right] = \left[ c_0, \ldots, c_i, c_{i+1}, c_{i+2}, c_{i+3} \ldots \right].
\end{equation}

In particular, in the example (\ref{ex1}) the formula (\ref{conj2}) gives for $\alpha=[0,2,3,\overline{2}]$
$$
\bar\alpha=[0,2,0,1,1,\bar{2}]=[0,3,1,\overline{2}].
$$
A more sophisticated example:
$$
\alpha=\frac{11523+\sqrt{15006}}{9222}=[1,3,1,4,\overline{7,2,3,9}],
$$
$$
\quad \bar\alpha=\frac{11523-\sqrt{15006}}{9222}=[1,3,0,1,4, \overline{3,2,7,9}] =[1,4,4,\overline{3,2,7,9}].
$$

\begin{proof} It is enough to consider the case $k=1$. Let us first define
\begin{equation*}
\beta = \left[\overline{b_1, \ldots, b_{l-1}, b_l}\right], \quad
\tilde{\beta} = \left[\overline{b_{l-1}, b_{l-2}, \ldots, b_1, b_l}\right].
\end{equation*}
Then, by definition,
\begin{equation*}
\alpha = \left[a_0, a_1, \overline{b_1, \ldots, b_l}\right] = a_0 + \cfrac{1}{a_1 
          + \cfrac{1}{\beta} },
\end{equation*}
so that
\begin{equation*}
\overline{\alpha} = a_0 + \cfrac{1}{a_1 
          + \cfrac{1}{\overline{\beta}} }:=[a_0,a_1, \overline{\beta}].
\end{equation*}
Using the Galois result and identity  (\ref{negid}) we have
$$
\overline{\beta} = - \left[0, \overline{b_l, \ldots, b_1} \right] = \left[-1, 1, b_l - 1, \tilde\beta \right].
$$
where 
$\tilde\beta=[\overline{b_{l-1}, \ldots, b_1, b_l}].$

We can now directly compute
$$
\overline{\alpha} = [a_0,a_1, \overline{\beta}]=\left[a_0,a_1,-1, 1, b_l - 1, \tilde\beta \right]
$$
$$
= \left[a_0,a_1,-1, 1+ \frac{\tilde\beta}{\left(b_l - 1\right)\tilde{\beta} + 1}  \right]
= \left[a_0,a_1,-1+\frac{\left(b_l - 1\right)\tilde{\beta} + 1}{b_l \tilde{\beta} + 1}    \right]
$$
$$
=\left[a_0, a_1 
          + \frac{b_l \tilde{\beta} + 1}{-\tilde{\beta} }   \right]=a_0 - \frac{\tilde{\beta}}{\left(b_l - a_1\right)\tilde{\beta} + 1 } 
$$

$$
= a_0 - 1 + \cfrac{1}{ 
           \cfrac{\left(b_l - a_1\right)\tilde{\beta} + 1}{\tilde{\beta} \left(b_l - a_1 - 1 \right) + 1 } } \\
= a_0 - 1 + \cfrac{1}{1 
          + \cfrac{\tilde{\beta}}{\tilde{\beta} \left(b_l - a_1 - 1 \right) + 1 } }
          $$
          $$
= a_0 - 1 + \cfrac{1}{1 
          + \cfrac{1}{\left(b_l - a_1 - 1 \right)  
					+ \cfrac{1}{\tilde{\beta}} } } \\
= \left[ a_0 - 1 , 1, b_l - a_1 - 1, \overline{b_{l-1}, \ldots, b_0, b_l} \right].
$$
This completes the case $a_k<b_l.$ In the case $a_k>b_l$
we have  by \eqref{eq:onezero},
\begin{equation*}
\alpha = \left[a_0, \ldots, a_k - b_l, 0, \overline{b_l, b_1, \ldots, b_{l-1}}\right],
\end{equation*}
and the result then immediately follows from the previous case with $a_k$ replaced by $0$ and $a_{k-1}$ replaced by $a_k - b_l$.
\end{proof}


Note that formulae (\ref{conj1}),(\ref{conj2}) determine involutions since for $a_k<b_l$ we have
\begin{align*}
\overline{\overline{\alpha}} &= \left[a_0, \ldots, a_{k-1}-1, 1-1, 1, b_l - \left( b_l - a_k - 1 \right) - 1 , \overline{b_1, \ldots, b_{l-1}, b_l}\right] \\
&= \left[a_0, \ldots, a_{k-1}-1, 0, 1, a_k , \overline{b_1, \ldots, b_{l-1}, b_l}\right] \\
&= \left[a_0, \ldots, a_{k-1}, a_k , \overline{b_1, \ldots, b_{l-1}, b_l}\right] = \alpha,
\end{align*}
while otherwise
\begin{align*}
\overline{\overline{\alpha}} &= \left[a_0, \ldots, a_{k-1}, a_k - b_l -1, 1-1, 1, b_{l-1}-1 -\left(b_{l-1}-1\right), \overline{b_{l}, b_1, \ldots, b_{l-1}}\right] \\
&= \left[a_0, \ldots, a_{k-1}, a_k - b_l -1, 0, 1, 0, \overline{b_{l}, b_1, \ldots, b_{l-1}}\right] \\
&= \left[a_0, \ldots, a_{k-1}, a_k- b_l, 0, \overline{b_{l}, b_1, \ldots, b_{l-1}}\right] = \alpha.
\end{align*}

Let us consider now the special case $k=0.$

\begin{prop} \label{thm:riverconj} Let $\alpha=[a_0, \overline{b_1,\dots, b_l}]$ with $a_0 < b_l$. Then the conjugate $\overline{\alpha}$ can be given as the negative continued fraction expansion
\begin{equation}
\label{conj3}
\overline{\alpha} = - \left[ b_l - a_0, \overline{b_l, \ldots, b_1} \right].
\end{equation}
When $a_0 > b_l$ we have 
\begin{equation}
\label{conj4}
\overline{\alpha}=[a_0-b_l-1,1,b_l-1, \overline{b_{l-1}, b_{l-2},\dots, b_1, b_l}].
\end{equation}
\end{prop}
\begin{proof}
We have
\begin{equation*}
\left[a_0, \overline{b_1, \ldots, b_n}\right] = a_0 + \left[ 0, \overline{b_0, \ldots, b_n}\right],
\end{equation*}
so by the Galois result
\begin{align*}
\overline{a_0 + \left[ 0, \overline{b_1, \ldots, b_l}\right]} &= a_0 - \left[\overline{b_l, \ldots, b_1} \right] \\
&= - \left[b_l - a_0, \overline{b_{l-1}, \ldots, b_1, b_l} \right].
\end{align*}

In the case when $a_0 > b_l$, we can rewrite 
\begin{equation*}
\left[a_0, \overline{b_1, \ldots, b_l} \right] = \left[a_0 - b_l, 0, \overline{b_l, b_1 \ldots, b_{l-1}} \right]
\end{equation*}
and proceed as in the case when $k\geq 1$ to get the formula (\ref{conj4}).
\end{proof}

Note that only in the Galois pure periodic case and in the case with $k=0, a_0 < b_l$ do we have to use negative continued fractions. We will see now that
these are the only cases when the initial position is on the Conway river.

\section{Paths to Conway river}

Let 
\begin{equation*}
Q(x, y) = ax^2 + hxy + by^2
\end{equation*}
be an indefinite integer binary quadratic form not representing zero, meaning that
$Q(x,y) \neq 0$ for all $(x,y)\in \mathbb Z^2\setminus (0,0).$ Equivalently, this means that the discriminant of the form
$
D=h^2-4ab
$
is positive, but not a total square.

In that case Conway \cite{Conway} showed that on the topograph of $Q$ positive and negative values are separated by a unique periodic river
(see Fig. 4). 



We are now going to explain how the continued fractions determine the path from the initial vertex $V_Q$ on the topograph with $a=Q(1,0),\, b=Q(0,1),\,c=Q(1,1)=a+b+h$ to the Conway river and the relation with the Galois result.

Let us assume for simplicity that all the values $a,b,c$ are of the same sign (say, positive), otherwise $V_Q$ is already on the river. Let us assume also that $h<0$, so we are going down to the river.

Let $\alpha, \bar\alpha$ be the real roots of the quadratic equation
\begin{equation} \label{eq:root}
Q(\alpha, 1) = a \alpha^2+h \alpha + b=0.
\end{equation}
Again for simplicity let us assume that
$$
\alpha= \frac{-h + \sqrt{D}}{2a}, \quad D=h^2-4ab
$$
is the dominant (maximal modulus) root.
Let 
\begin{equation} \label{eq:alphaexp}
\alpha = [a_0, a_1, \ldots, a_k, \overline{b_1, \ldots, b_l}]
\end{equation}
be the continued fraction expansion of $\alpha.$ 

From the general theory of the Farey tree and continued fractions we have the following result.

\begin{prop}
The continued fraction expansion (\ref{eq:alphaexp}) describes the unique path $\gamma_\alpha$ going from $V_Q$ to the corresponding Conway river, and then along the river towards $\alpha$. A similar path $\gamma_{\bar\alpha}$ leading to $\bar\alpha$ is described by the Propositions 1 and 2 above.

The Conway river of $Q$ is an infinite in both directions path from $\overline{\alpha}$ to $\alpha$ described by the periodic part $\left[\overline{b_1, \ldots, b_l } \right]$ of the expansion.

The finite path $\pi$ from $V_Q$ to the Conway river can be given for $k\geq 1$ by
\begin{equation*}
\pi = \left\{
  \begin{array}{lr}
    \left[ a_0, \ldots, a_k - b_l -1 \right] & {\text if} \,\, a_k > b_l\\
    \left[ a_0, \ldots, a_{k-1} -1 \right] & {\text if} \,\, a_k < b_l
  \end{array}
\right. 
\end{equation*}
and for $k=0$ by
\begin{equation*}
\pi = \left\{
  \begin{array}{lr}
    a_0 - b_l& {\text if} \,\, a_0 > b_l\\
    \emptyset &{\text if} \,\,  a_0 < b_l
  \end{array}
\right. .
\end{equation*}
\end{prop}

This also gives a geometric interpretation of our formulas from the previous section, which is
 illustrated in Fig. 6 in the example of $Q=17x^2-12xy+2y^2$ with 
$$
\alpha=\frac{6+\sqrt{2}}{17}=[0,2,3,\overline{2}], \quad \bar\alpha=\frac{6-\sqrt{2}}{17}=[0,3,1,\overline{2}].
$$

 \begin{figure}[h]
\begin{center}
\includegraphics[height=80mm]{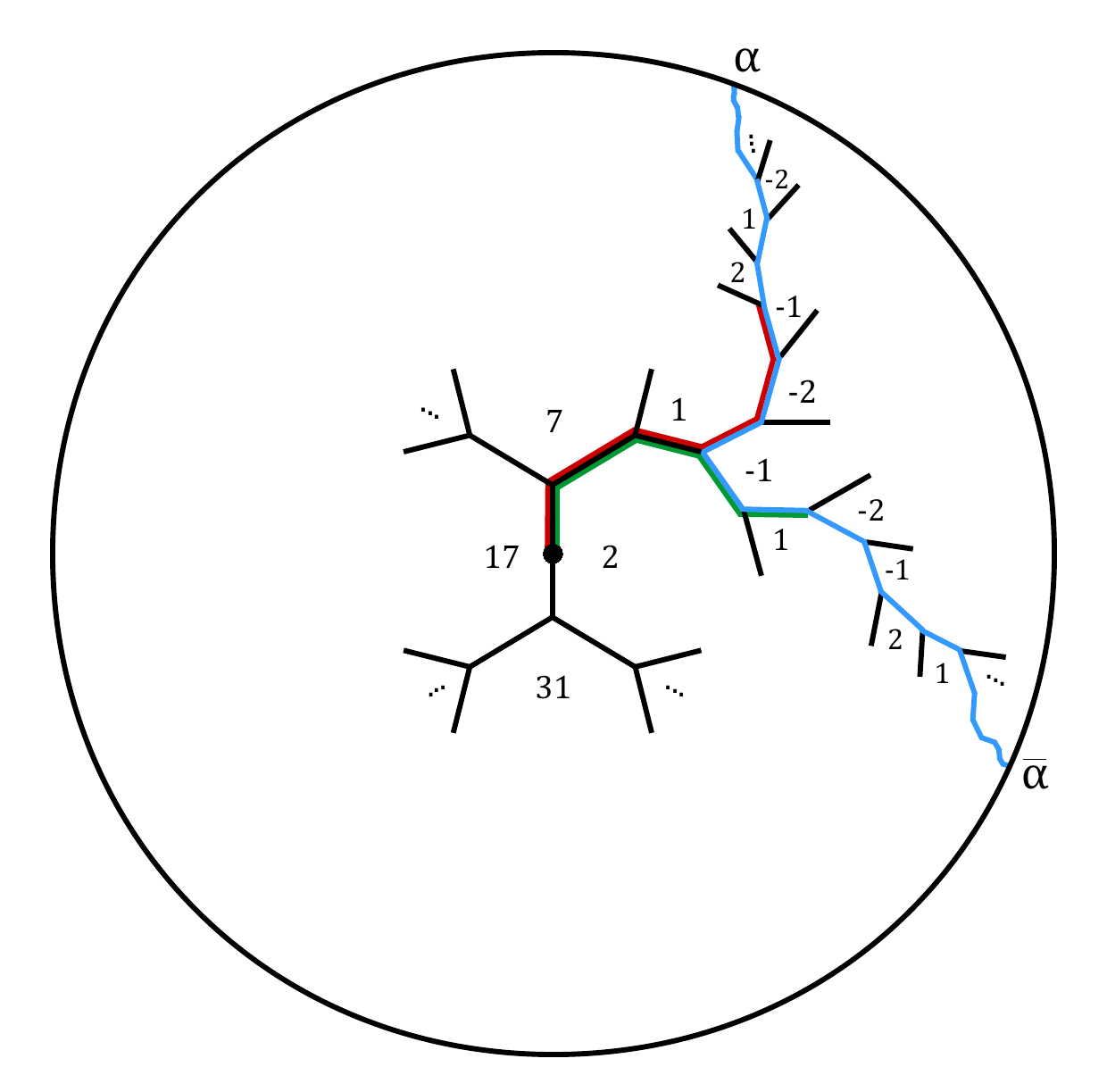} 
\caption{\small Paths to $\alpha$ and $\bar\alpha$ and Conway river for $Q=17x^2-12xy+2y^2.$}
\end{center}
\end{figure}

Let us describe now the special (Galois) case when $\alpha$ is a pure periodic continued fraction.

Let us call the quadratic form $Q(x, y) = ax^2 + hxy + by^2$ a {\it Galois form} if the continued fraction expansion of the dominant root of
$Q(\alpha,1)=0$ is pure periodic. By the Galois theorem this is equivalent to the conditions $\alpha>1, \, -1<\bar\alpha<0.$

\begin{prop}
An integer binary quadratic form $Q$ is Galois if and only if $Q(1,0) = a$, $Q(0,1) = b$, $Q(1,1) = c$ satisfy
$$
ab < 0, \,\,
ac < 0, \,\,
ac' = a(2a+2b-c) >0.
$$
\end{prop}

The proof is obvious geometrically: $\alpha > 1$ and $-1 < \overline{\alpha} < 0$ is true if and only if the Conway river travels along the edges separating $a$ and $c'$ from $b$ and $c$, see Fig. 7.

	\begin{figure}[h]
	\centering
	\includegraphics[scale=0.3]{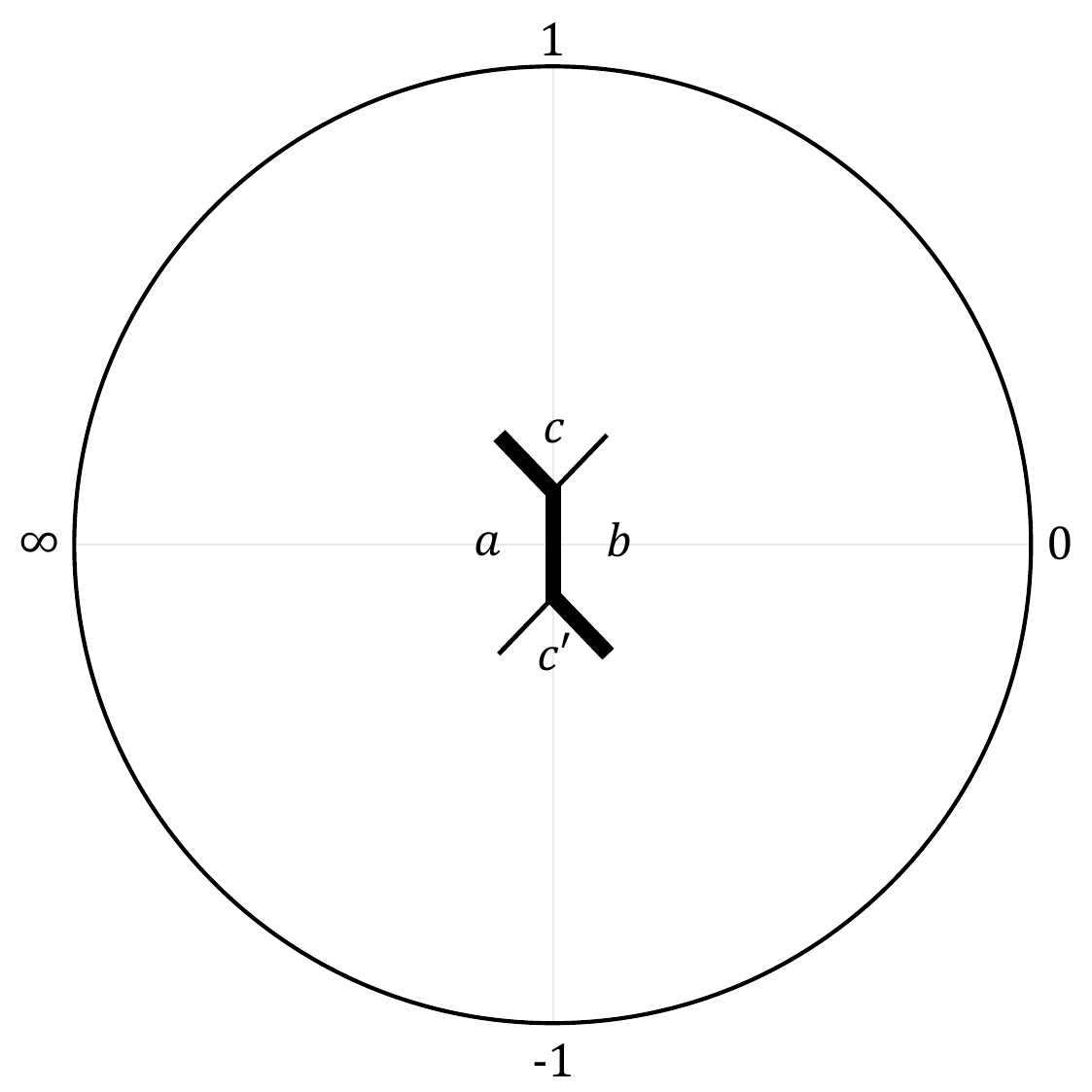} \label{fig:basepluslimit}
	\caption{Galois form $Q$ on the Conway river.}
	\end{figure}
	
	An example of the Galois form is the ``golden" form $Q=x^2-xy-y^2$ corresponding to $a=1=c', b=c=-1$ with
	$$\alpha=\varphi=\frac{1+\sqrt{5}}{2} =[\bar 1], \,\bar\alpha=\frac{1-\sqrt{5}}{2}=-[0,\bar 1].$$

\section{Lyapunov exponents for values of binary forms}

Now we are ready to state and prove our main results.

Let $Q(x,y)=ax^2+hxy+by^2$ be an integer binary quadratic form (definite or indefinite). 

If the form $Q$ is indefinite, then we first assume that $Q$ does not represent zero in the sense that
$Q(x,y)\neq 0$ for all $(x,y)\in \mathbb Z^2\setminus(0,0).$
In that case the two roots $\alpha_\pm=\alpha,\bar{\alpha}$ of the quadratic equation
$$
Q(\xi,1)=0
$$
are real quadratic irrationals, corresponding to the ends of the Conway river.

To study the growth of the values of $Q$ along the path $\gamma_\xi$ we define the corresponding Lyapunov exponent as
\beq{defL}
\Lambda_Q(\xi)=\limsup_{n\to\infty}\frac{\ln |Q_n(\xi)|}{n}, \quad |Q_n(\xi)| = \max(|a_n(\xi)|, |b_n(\xi)|, |c_n(\xi)|),
\eeq
where $a_n(\xi), b_n(\xi), c_n(\xi)$ are the values of $Q$ at the $n$-th superbase on the path $\gamma_\xi.$

Let $\Lambda(\xi)$ be the Lyapunov exponent of the monoid $SL_2(\mathbb N)$ introduced in \cite{SV}:
\beq{defmonoid}
\Lambda(\xi)=\limsup_{n\to\infty}\frac{\ln \rho(A_n(\xi))}{n},
\eeq
where $A_n(\xi)=A_n(\gamma_\xi)$ are the matrices (\ref{mata}) and $\rho(A)$ is the spectral radius of the matrix $A.$

\begin{Theorem}
For the definite integer binary quadratic forms $Q$ the Lyapunov exponent
$$
\Lambda_Q(\xi)=2\Lambda(\xi).
$$
For the indefinite integer binary quadratic forms $Q$ not representing 0 we have
\begin{equation*}
\Lambda_Q(\xi) = \left\{
  \begin{array}{lr}
   2\Lambda(\xi)& {\text if} \,\, \xi \neq \alpha_\pm\\
    0 &{\text if} \,\,  \xi = \alpha_\pm.
  \end{array}
\right. 
\end{equation*}
In other words, the only two exceptional paths with zero growth are those leading to two ends of the Conway river of $Q.$
\end{Theorem}

\begin{proof}
Let us introduce first for the form $Q(x,y)=ax^2+hxy+by^2$ its matrix defined by
$$
Q(x,y)=\begin{pmatrix}
  x & y 
\end{pmatrix}\begin{pmatrix}
  a & \frac{h}{2} \\
  \frac{h}{2} & b \\
\end{pmatrix}\begin{pmatrix}
  x \\
  y \\
\end{pmatrix},
$$
which, slightly abusing the notation, we also denote by $Q$:
$$
Q=\begin{pmatrix}
  a & \frac{h}{2} \\
  \frac{h}{2} & b \\
\end{pmatrix}.
$$
Its determinant $$\det Q=ab-\frac{h^2}{4}=-\frac{D}{4}$$ is minus a quarter of the discriminant $D=h^2-4ab$ of the quadratic equation $Q(\xi,1)=0.$

The action of the group $SL_2(\mathbb Z)$ on $Q$ is defined in the matrix form by
$$Q \to Q'=A^t Q A,$$ where
\beq{A}
A= \begin{pmatrix}
  p & q \\
  r & s \\
\end{pmatrix} \in SL_2(\mathbb Z).
\eeq
This determines a 3-dimensional representation $A \to \hat A$ of $SL_2(\mathbb Z)$, when $A$ acts on the coefficients $(a,h,b)$ of the form $Q$ by 
$$
\hat A=\begin{pmatrix}
  p^2 & 2pr & r^2 \\
  pq & pr+qs & rs \\
  q^2 & 2qs & s^2 \\
\end{pmatrix}.
$$
If the eigenvalues of the matrix $A$ are $\lambda$ and $\lambda^{-1},$ then the eigenvalues of $\hat A$ are $\lambda^2$,  $\lambda^{-2}$ and 1.

Consider first the case when the coefficients of the form $a,b,h$ are all positive. In that case Conway's Climbing Lemma \cite{Conway} guarantees the permanent growth whichever upward path we choose (see Fig. 8). 

\begin{figure}[h]
	\centering
	\includegraphics[height=40mm]{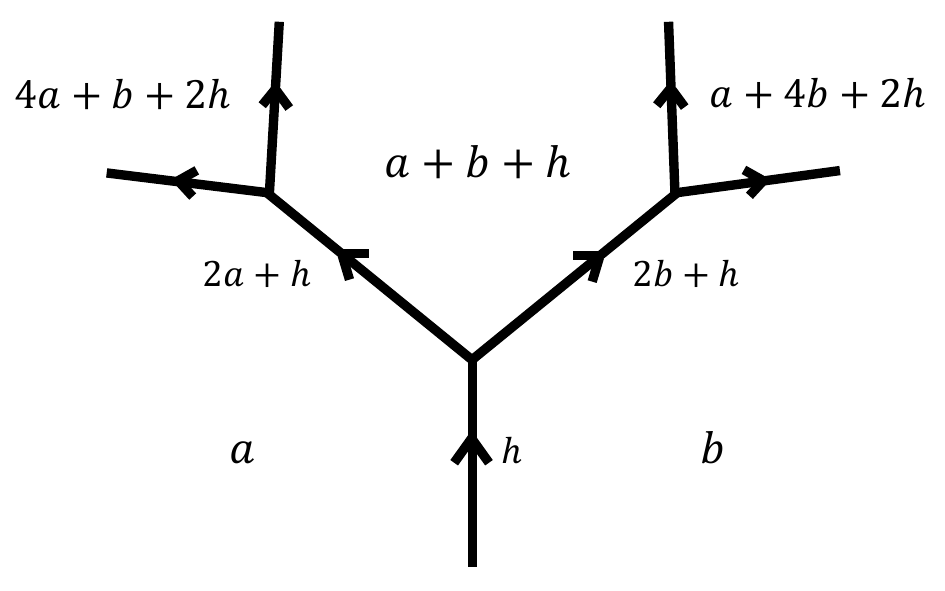} \label{fig:CL}
	\caption{Climbing Lemma.}
	\end{figure}

We claim that the corresponding growth is given by $\Lambda_Q(\xi)=2\Lambda(\xi).$

To prove this, introduce two norms on binary forms $Q=ax^2+hxy+by^2$. The first norm is the one used in the definition of $\Lambda_Q$:
\beq{norm1}
|Q|=\max(|a|,|b|,|c|), \,\, c=a+b+h.
\eeq
The second norm is 
\beq{norm2}
|Q|_h=\max(|a|,|b|,|h|).
\eeq
Since any two norms in a finite-dimensional space are equivalent, we have for all $Q$
$$
c_1|Q|_h\leq|Q|\leq c_2|Q|_h
$$
for some positive constants $c_1,c_2.$
This means that we can replace $|Q|$ by $|Q|_h$ in the definition 
$$
\Lambda_Q(\xi)=\limsup_{n\to\infty}\frac{\ln |Q_n(\xi)|}{n}=\limsup_{n\to\infty}\frac{\ln |Q_n(\xi)|_h}{n}.
$$

Now let $\gamma_\xi$ be a path and 
$$A_n(\xi)= \begin{pmatrix}
  p_n & q_n \\
  r_n & s_n \\
\end{pmatrix} \in SL_2(\mathbb Z).
$$ 
be the corresponding matrices (\ref{mata}), so that the matrices of the forms along the path are
$$
Q_n(\xi)=A_n(\xi)^TQ A_n(\xi),
$$
or, explicitly
\beq{3d}
\begin{pmatrix}
  a_n \\
  h_n \\
  b_n\\
\end{pmatrix}=\begin{pmatrix}
  p_n^2 & 2p_nr_n & r_n^2 \\
  p_nq_n & p_nr_n+q_ns_n & r_ns_n \\
  q_n^2 & 2q_ns_n & s_n^2 \\
\end{pmatrix}
\begin{pmatrix}
  a \\
  h \\
  b\\
\end{pmatrix}.
\eeq

Assuming without loss of generality that $\xi \in [0,1]$ we have
$$
\Lambda(\xi)=\limsup_{n\to\infty}\frac{\ln w_n(\xi)}{n},
$$
where $w_n=r_n+s_n$ (see Proposition 1 in \cite{SV}).

Consider the cases $\xi=0$ and $\xi>0$ separately.
When $\xi=0$ we have
$$A_n(0)= \begin{pmatrix}
 1 & 0 \\
  n & 1 \\
\end{pmatrix},
$$ 
which determines the quadratic growth of $|Q_n|_h$ in $n$, and as a result
$$
\Lambda_Q(0)=0=2\Lambda(0).
$$

If $0<\xi\leq 1$ we have $p_n\leq r_n, q_n\leq s_n$ and thus from (\ref{3d})
$$
|Q_n(\xi)|_h \leq \max((p_n+r_n)^2, (p_n+s_n)(q_n+r_n), (q_n+s_n)^2)|Q|_h
$$
$$
\leq \max(4r_n^2, (r_n+s_n)^2, 4s_n^2)|Q|_h\leq 4(r_n+s_n)^2 |Q|_h=4w_n^2 |Q|_h.
$$
On the other hand, by assumption the initial $a,b,h$ are all positive, so, since they are also integer, $a,b,h\geq 1.$
From (\ref{3d}) then it follows that
$$
|Q_n(\xi)|_h\geq (p_n+r_n)^2+ (p_n+s_n)(q_n+r_n)+ (q_n+s_n)^2
$$
$$
>r_n^2+s_n^2 \geq \frac{1}{2}(r_n+s_n)^2=\frac{1}{2}w_n^2.
$$
Thus we have
$$
\frac{1}{2}w_n^2\leq|Q_n(\xi)|_h\leq4w_n^2 |Q|_h,
$$
which implies that
$$
\Lambda_Q(\xi)=\limsup_{n\to\infty}\frac{\ln |Q_n(\xi)|}{n}=\limsup_{n\to\infty}\frac{\ln w_n^2(\xi)}{n}=2\Lambda(\xi).
$$
This completes the proof in the growing case. 

Assume now that $a,b$ are positive, but $h<0$, so the values of $Q$ decrease.
Consider first the case when the form $Q$ is positive definite. Then all the values are positive, so it is clear that for any path at some point the growth will start again and we can repeat our arguments to get the claim.

However, if $Q$ is indefinite this is no longer true, since $Q$ can take negative values as well. By Conway's result \cite{Conway} for indefinite $Q$ not representing zero positive and negative values of $Q$ are separated by an infinite periodic river. 

There are three possibilities: either the path does not cross the river, it crosses the river (or starts on the river and then leaves it), or it is after some point stuck on the river forever.

If the path does not cross the river the values of $Q$ will remain positive and thus bounded from below, so at some point we will have growth and we repeat the arguments to prove the claim in this case as well.

If the path crosses the river then at some point all the values of $a_n, b_n, c_n$ will become negative, and we can repeat the arguments for $-Q$ to get the claim. If we start on the river (which means that $a$ and $b$ are of opposite signs), then after the point of departure we will have $a$ and $b$ of the same sign and can use the same arguments as before.

Finally, there are exactly two paths which are stuck on the river, corresponding to $\xi=\alpha_\pm.$ In that case we have no growth because of periodicity of the Conway river \cite{Conway} and thus 
$$
\Lambda_Q(\alpha_\pm)=0.
$$
Note that the corresponding $\Lambda(\alpha_\pm)\neq 0,$ so $\Lambda_Q(\xi)\neq 2\Lambda(\xi)$ in that case.
\end{proof}

Let us discuss now the remaining case: the indefinite forms representing zero and the semidefinite forms with zero discriminant.

\begin{prop}
For indefinite integer quadratic forms $Q$ representing zero $$
\Lambda_Q(\xi)=2\Lambda(\xi),\,\, \xi\in \mathbb RP^1.
$$
The same is true for the semidefinite integer forms.

\end{prop}

\begin{proof}
The topograph of the indefinite forms representing zero (with discriminant $D$ being total square) is described in Conway \cite{Conway}.
In that case we have two ``lakes" corresponding to zero values connected by a finite river separating positive and negative values of $Q,$ see Fig. 9.

\begin{figure}[h]
	\centering
	\includegraphics[height=40mm]{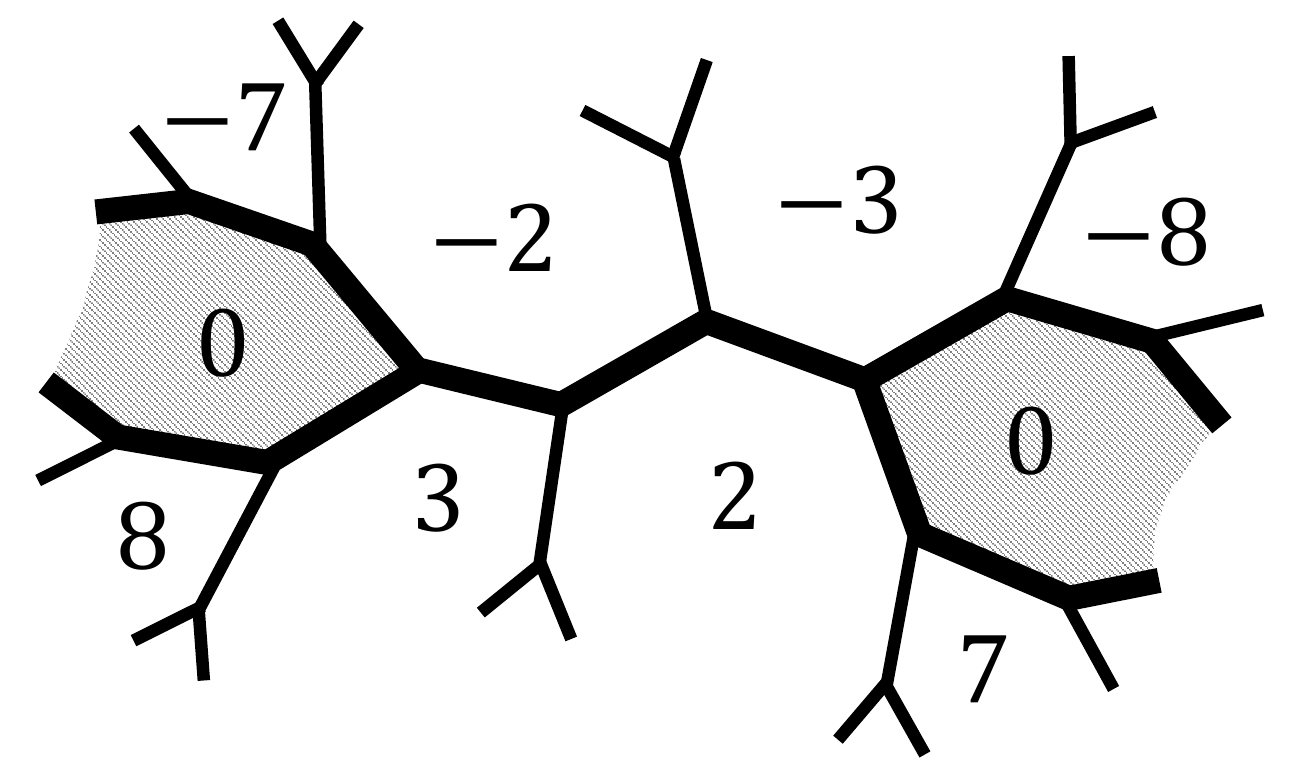} \label{fig:CL}
	\caption{Conway river and lakes for $Q=2x^2 - xy - 3y^2.$}
	\end{figure}

Note that modulo $GL_2(\mathbb Z)$ any such form can be reduced to
$$
Q_{red}=y(ax-by), \,\, a,b \in \mathbb Z_+, \, a>b.
$$
The sequence of turns of the Conway river, connecting the lakes, is determined by the finite continued fraction expansion
$$
\frac{a}{b}=[a_0,\dots,a_n].
$$
In particular, for $Q=2x^2 - xy - 3y^2$ the reduced form is $Q_{red}=y(5x-3y)$ and $\frac{5}{3}=[1,1,2]$, corresponding to the sequence of river's turns $R,L,R,R$ (see Fig. 9, where we start on the edge between 0 and 3 on the left lake).
In the exceptional case $b=0$ with $Q_{red}=axy$ the river disappears and two lakes are adjacent \cite{Conway}.

This picture agrees with the fact that the roots $\alpha_\pm$ of the corresponding equation
$$
Q(\alpha,1)=0
$$
are rational. A simple analysis shows that the previous arguments in this case work as well, with the only difference that now both
$\Lambda_Q(\alpha_\pm)$ and $\Lambda(\alpha_\pm)$ are zero, and thus equality $\Lambda_Q(\xi)=2\Lambda(\xi)$ holds for all $\xi$.
The same works for the semidefinite forms with $Q_{red}=ax^2$ as well.
\end{proof}

{\it Remark.} Most of these arguments work also for the binary forms $Q=ax^2+hxy+by^2$ with {\it real} coefficients $a,h,b$.
The Conway river is not necessarily periodic in this case and closely related to the Arnold sail in Klein's geometric approach to continued fractions \cite{Klein} (see \cite{SV1}).
However, in that case the values of the form along the Conway river may approach zero (see e.g. \cite{KW}), so the situation here is a bit more delicate.

\section{Concluding remarks}

Let us compare this with the growth of the Markov triples considered in \cite{SV}.
Markov triples are the positive solutions of the Markov equation \cite{Markov}
\beq{equa}
x^2+y^2+z^2-3xyz=0.
\eeq
They all can be found from the obvious one $(1,1,1)$ by compositions of Vieta involutions
\beq{invo}
(x,y,z) \rightarrow (x,y, 3xy-z)
\eeq
and permutations of $x,y,z.$

The right-hand side of the Markov equation
$$
F(x,y,z):=x^2+y^2+z^2-3xyz
$$
is cubic polynomial, which is {\it quadratic} in every variable, implying the Vieta involution.
It is also symmetric under the permutations of $x,y,z.$
One can consider also a version of the Markov equation
$$
F(x,y,z)=D,
$$
which has similar properties (see \cite{SV}).

We claim that this paper can be interpreted in a similar way for degree 2 polynomial
$$
B(x,y,z)=a^2+b^2+c^2-2ab-2ac-2bc.
$$
Indeed, the transformation $$(a,b,c)\rightarrow (a,b,c'=2a+2b-c)$$
is just the Vieta involution for the equation
\beq{B}
a^2+b^2+c^2-2ab-2ac-2bc=D.
\eeq
In fact, $B(a,b,c)$ is nothing other than the discriminant of the binary quadratic form $Q=ax^2+hxy+by^2:$
$$
D=h^2-4ab=(c-a-b)^2-4ab=a^2+b^2+c^2-2ab-2ac-2bc.
$$
Our results say that the Lyapunov exponents for the equation (\ref{B}) with $D<0$ are twice the Lyapunov exponents for the Markov dynamics.
When $D>0$ then the same is true with the exception of Conway river paths, for which the Lyapunov exponents vanish.

Note that the equation (\ref{B}) with $D<0$ determines the two-sheeted hyperboloid, while for $D>0$ this is a one-sheeted hyperboloid.
The semidefinite (degenerate) forms with $D=0$ correspond to the cone.

We would like to mention Klein's correspondence \cite{Klein} between indefinite binary quadratic forms and geodesics in hyperbolic plane, which we have learnt recently from Boris Springborn \cite{Springborn}.
As we have just seen the projectivized vector space of real binary quadratic forms is a real projective plane with the degenerate
forms forming a conic section determined by (\ref{B}) with $D=0.$ Definite forms correspond to the points inside
this conic, which in the Klein model represents the points of the hyperbolic plane. The indefinite forms correspond
to the points outside the conic, which by polarity are represented by the hyperbolic lines. On the binary tree these lines are nothing other than Conway rivers.


The Galois result is related to the reduction theory for the integer binary quadratic forms and the elements of $SL_2(\mathbb Z)$ going back to Gauss (see Chapter 7 in \cite{Karp} and references therein). In particular, we would like to mention an interesting paper by Manin and Marcolli \cite{MM}, where this theory is linked with non-commutative geometry and the quotient $\mathbb{R}P^1/PGL_2(\mathbb Z)$ is interpreted as a non-commutative version of the modular curve. It would be interesting to understand from this point of view the role of our function $\Lambda(\xi)$, which is naturally defined on this space.

Finally, a natural question is about the growth of values of binary forms $F$ of degree $d>2$. The corresponding action of $SL_2(\mathbb Z)$ is the restriction of the standard irreducible representation of $SL_2(\mathbb R)$ of dimension $d+1.$
This implies that generically the corresponding growth of the values of $F$ is 
$$
\Lambda_F(\xi)=d \Lambda(\xi),
$$
but the precise description of all exceptions is an interesting question, open even for binary cubic forms.
Note that in this case the values of $F$ on a superbase do not define $F$ uniquely and as a result we do not have local rules similar to the paralellogram rule.

\section{Acknowledgements}

We are very grateful to Alexey Bolsinov for numerous helpful discussions, to Oleg Karpenkov for explaining us his nice results on geometry of continued fractions \cite{Karp}, to Boris Springborn for sending us a preliminary version of his very interesting work \cite{Springborn} and to the referee for very helpful and constructive comments.

The work of K.S. was supported by the EPSRC as part of PhD study at Loughborough.

\end{document}